\documentclass[11pt]{amsart}

\usepackage[margin=1in]{geometry}
\usepackage{amsmath,amssymb,amsthm,mathrsfs}

\theoremstyle{plain}
\newtheorem{theorem}{Theorem}
\newtheorem{proposition}[theorem]{Proposition}

\theoremstyle{remark}

\newcommand{\bZ}{\mathbb{Z}}
\newcommand{\bR}{\mathbb{R}}
\newcommand{\bQ}{\mathbb{Q}}

\begin{document}

\begin{quote}
\begin{center}
\emph{The second of a pair of papers in memory of Professor Michio Ozeki.}
\end{center}
\end{quote}

\title[Second-shell generation for extremal Type~II lattices of rank~$72$]
{Extremal Type~II lattices of rank~$72$ are generated\\
by their second shell}

\author[Scott Duke Kominers]{Scott Duke Kominers}
\address{Harvard Business School; Department of Economics and Center of
Mathematical Sciences and Applications, Harvard University; and a16z crypto}
\email{kominers@fas.harvard.edu}
\thanks{I used LLMs to assist with computations in the preparation of this
article, particularly GPT-5.4 Pro, Gemini 3.1 Pro, and Claude 4.6 Opus
(all accessed via Poe with the support of Quora, where I am an advisor). I
especially appreciate helpful comments from Manabu Oura and a thorough review
from Refine.ink. The problem, methods, and eventual written form are my own;
and of course any errors remain my responsibility. This work was conducted
while I was visiting the Technological Innovation, Entrepreneurship, and
Strategic Management (TIES) Group at the MIT Sloan School of Management; I
greatly appreciate their hospitality.}


\begin{abstract}
We show that if $L$ is an extremal Type~II lattice of rank~$72$, then
$L$ is generated by its vectors of norm~$10$. The proof determines the
full inner product distribution between the shells of norms~$8$ and~$10$
using the spherical $11$-design property of the norm-$10$ shell.
\end{abstract}

\maketitle

\section{Introduction}

A \emph{lattice} of rank~$n$ is a free $\bZ$-module of rank~$n$
equipped with a positive-definite inner product
$\langle\cdot,\cdot\rangle\colon L\times L\to\bR$. For $x\in L$, we
write $\langle x,x\rangle$ for the \emph{norm} of~$x$. A lattice is
\emph{even} if $\langle x,x\rangle\in 2\bZ$ for all $x\in L$, and
\emph{unimodular} (or \emph{self-dual}) if $L=L^*$, where $L^*$ is the
dual lattice. An even unimodular lattice is said to be of
\emph{Type~II}, and the rank of such a lattice is divisible by~$8$
(see, e.g.,~\cite{Ser73}).

Mallows, Odlyzko, and Sloane~\cite{MOS} proved that the minimum nonzero
norm $\min(L)$ of a Type~II lattice of rank~$n$ satisfies
\[
  \min(L) \le 2\lfloor n/24\rfloor+2;
\]
Type~II lattices attaining this bound are called \emph{extremal}. For
each $j\ge 0$, we write
\[
  \Lambda_j(L)=\{x\in L:\langle x,x\rangle=j\}
\]
for the shell of norm-$j$ vectors, and write $\mathcal{L}_j(L)$ for the
sublattice generated by~$\Lambda_j(L)$.

Extremal Type~II lattices exhibit remarkable geometric regularity. Their
nonempty shells are highly symmetric spherical designs~\cite{Venkov84spherical},
and this rigidity leads to strong configuration results governing the
interaction of vectors of different norms. In particular,
Venkov~\cite{Ven84} showed that an extremal Type~II lattice of rank~$32$ 
is generated by its vectors of minimal norm (see also~\cite{Oze86a}), 
and Ozeki~\cite{Oze86b} established the analogous result in rank~$48$. 
Kominers~\cite{Kom09} extended Ozeki's methods to ranks~$56$, $72$, 
and~$96$, proving in particular:

\begin{theorem}[\cite{Kom09}]\label{thm:prior}
  If $L$ is an extremal Type~II lattice of rank~$72$, then $L$ is
  generated by its vectors of minimal norm $\min(L)=8$, i.e.,
  $\mathcal{L}_8(L)=L$.
\end{theorem}

Meanwhile, Ozeki~\cite{Oze89} observed that there exist extremal
Type~II lattices of rank~$40$ that are \textit{not} generated by their
minimal vectors alone, while at the same time showing that every
extremal Type~II lattice of rank~$40$ is generated by the shells of
norms $\min(L)$ and $\min(L)+2$. Kominers and Abel~\cite{KA08} proved
the analogous two-shell generation result in ranks~$80$ and~$120$.
Elkies and Kominers~\cite{EK09} then sharpened the conclusion in
ranks~$40$ and~$80$ by showing that in those dimensions, the shell of
norm $\min(L)+2$ alone generates the lattice.

In this paper we prove the analogous second-shell generation result in
rank~$72$:

\begin{theorem}\label{thm:main}
  If $L$ is an extremal Type~II lattice of rank~$72$, then $L$ is
  generated by its vectors of norm $\min(L)+2=10$, i.e., 
	$\mathcal{L}_{10}(L)=L$.
\end{theorem}

The existence of an extremal Type~II lattice of rank~$72$ was
established by Nebe~\cite{Nebe}, who constructed such a lattice as a
Hermitian tensor product of the Barnes lattice and the Leech lattice
over the ring of integers in~$\bQ(\sqrt{-7})$.

The key point in our argument is that, in rank~$72$, the relevant
cross-shell configuration is determined outright by the spherical design
identities. Fix a minimal vector $x_0\in\Lambda_8(L)$, and for each
integer $j$ consider the number of vectors $y\in\Lambda_{10}(L)$ with
$\langle x_0,y\rangle=j$. In ranks~$40$ and~$80$, the available
relations coming from spherical design theory and weighted theta series
do not determine these multiplicities \emph{a priori}; Elkies and
Kominers~\cite{EK09} therefore proceeded by contradiction, forcing one
critical multiplicity to vanish and then showing that the resulting
projective solution includes a negative entry. In rank~$72$, by
contrast, the norm-$10$ shell is a spherical $11$-design, and the
resulting even moments of degrees $2$, $4$, $6$, $8$, and~$10$ give
enough equations to determine all possible nonzero inner-product
multiplicities uniquely.

In fact, the proof yields more than Theorem~\ref{thm:main} alone: it
determines the full inner product distribution between the shells
$\Lambda_8(L)$ and $\Lambda_{10}(L)$. In particular, we find that for
every $x_0\in\Lambda_8(L)$ there exists $y\in\Lambda_{10}(L)$ with
\[
  \langle x_0,y\rangle=4.
\]
It follows that $y-x_0\in\Lambda_{10}(L)$, so every minimal vector of
$L$ is the difference of two vectors of norm~$10$. Since
$\mathcal{L}_8(L)=L$ by Theorem~\ref{thm:prior}, it follows that
$\mathcal{L}_{10}(L)=L$.

\section{Proof of Theorem~\ref{thm:main}}\label{sec:proof}

Let $L$ be an extremal Type~II lattice of rank~$72$, so that
$\min(L)=8$. We write $a(m)$ for the number of vectors of norm~$2m$ in
$L$; equivalently, $a(m)$ is the $m$-th Fourier coefficient of the
\emph{theta series}
\[
  \Theta_L(z)=\sum_{x\in L}q^{\langle x,x\rangle/2}
  \qquad (q=e^{2\pi iz}).
\]

\subsection*{Theta series and shell sizes}

Because $L$ is Type~II, its theta series $\Theta_L(z)$ is a modular
form of weight~$36$ for $\mathrm{SL}_2(\bZ)$. The space
$M_{36}(\mathrm{SL}_2(\bZ))$ of such modular forms has basis
\[
  E_4^9,\qquad
  \Delta E_4^6,\qquad
  \Delta^2 E_4^3,\qquad
  \Delta^3,
\]
where $E_4$ is the normalized Eisenstein series of weight~$4$ and
$\Delta$ is the discriminant cusp form. Together with the constant
term~$1$, the extremality conditions $a(1)=a(2)=a(3)=0$ determine
$\Theta_L$ uniquely~\cite{SPLAG}:
\[
  \Theta_L
  =E_4^9-2160\,\Delta E_4^6
   +965520\,\Delta^2 E_4^3
   -27302400\,\Delta^3.
\]
Expanding at the cusp gives
\begin{equation}\label{eq:shells}
  \Theta_L
  =1+6{,}218{,}175{,}600\,q^4
   +15{,}281{,}788{,}354{,}560\,q^5
   +O(q^6).
\end{equation}
Hence, we have
\[
  |\Lambda_8|=a(4)=6{,}218{,}175{,}600,
  \qquad
  |\Lambda_{10}|=a(5)=15{,}281{,}788{,}354{,}560;
\]
in particular, $\Lambda_{10}\neq\varnothing$.\footnote{Here and hereafter,
since $L$ is fixed throughout, we abuse notation slightly by writing
``$\Lambda_8$'' for $\Lambda_8(L)$ and ``$\Lambda_{10}$'' for
$\Lambda_{10}(L)$ when doing so will not introduce confusion.}

\subsection*{Cross-shell inner product constraints}

Fix $x_0\in\Lambda_8(L)$. Following~\cite{EK09,KA08}, define
\[
  M_j(L;x_0)
  =
  \bigl|\{y\in\Lambda_{10}(L):\langle x_0,y\rangle=j\}\bigr|
  \qquad (j\in\bZ).
\]
By the symmetry $y\mapsto -y$, we have
$M_{-j}(L;x_0)=M_j(L;x_0)$ for all~$j$.

Let $y\in\Lambda_{10}(L)$ and put $j=\langle x_0,y\rangle$. Choose
$\varepsilon\in\{\pm1\}$ so that $\varepsilon j=|j|$. Then
\[
  \langle y-\varepsilon x_0,\,y-\varepsilon x_0\rangle
  =10-2|j|+8
  =18-2|j|.
\]
Since $y-\varepsilon x_0\neq 0$ otherwise $y=\varepsilon x_0$, which is
impossible because their norms are different, extremality implies that
$18-2|j|\ge 8$. Hence $|j|\le 5$. Therefore the possible inner products
are
\[
  j\in\{0,\pm1,\pm2,\pm3,\pm4,\pm5\}.
\]

\subsection*{Moment equations from the $11$-design property}

Venkov~\cite{Venkov84spherical} proved that every nonempty shell of an
extremal Type~II lattice of rank~$n$ is a spherical $t$-design, in fact
a spherical $(t\frac{1}{2})$-design, where $t=11,7,3$ according as
$n\equiv 0,8,16\pmod{24}$. Thus, in particular, every nonempty shell of
an extremal Type~II lattice of rank~$72$ is a spherical $11$-design.

We apply the $11$-design property of the shell
$\Lambda_{10}/\sqrt{10}\subset S^{71}$, where $S^{71}$ denotes the
$71$-sphere. For a spherical $t$-design $Z\subset S^{n-1}$ and any unit
vector $\hat e\in S^{n-1}$, the even moment formula~\cite{DGS77} gives
\begin{equation}\label{eq:designformula}
  \frac{1}{|Z|}\sum_{z\in Z}(\hat e\cdot z)^{2k}
  =
  \frac{(2k-1)!!}{n(n+2)\cdots(n+2k-2)}
\end{equation}
for $1\le 2k\le t$. Taking $\hat e=x_0/\sqrt{8}$ and
$Z=\Lambda_{10}/\sqrt{10}$ with $n=72$, we have
$\hat e\cdot z=\langle x_0,y\rangle/\sqrt{80}$, and thus obtain
\begin{equation}\label{eq:moments}
  \sigma_{2k}
  :=
  \sum_{y\in\Lambda_{10}}\langle x_0,y\rangle^{2k}
  =
  |\Lambda_{10}|\cdot 80^k\cdot
  \frac{(2k-1)!!}{\displaystyle\prod_{i=0}^{k-1}(72+2i)}
\end{equation}
for $k=1,2,3,4,5$.

The right-hand side of~\eqref{eq:moments} is independent of the
particular choice of $x_0\in\Lambda_8(L)$; hence the moments
$\sigma_{2k}$ are the same for every such $x_0$. Substituting in the
value of $|\Lambda_{10}|$ from~\eqref{eq:shells}, we find
\begin{equation}\label{eq:momentvalues}
\begin{aligned}
  \sigma_2    &= 16{,}979{,}764{,}838{,}400,\\
  \sigma_4    &= 55{,}069{,}507{,}584{,}000,\\
  \sigma_6    &= 289{,}839{,}513{,}600{,}000,\\
  \sigma_8    &= 2{,}080{,}899{,}072{,}000{,}000,\\
  \sigma_{10} &= 18{,}728{,}091{,}648{,}000{,}000.
\end{aligned}
\end{equation}

\subsection*{The linear system}

We write $M_j=M_j(L;x_0)$ for brevity. Using the fact that
$M_{-j}=M_j$, the moments~\eqref{eq:moments} decompose as
\[
  \sigma_{2k}=2\sum_{j=1}^{5} M_j j^{2k}
  \qquad (k=1,\dots,5).
\]
Equivalently, using~\eqref{eq:momentvalues}, we have
\[
  \begin{pmatrix}
    1 & 4 & 9 & 16 & 25 \\
    1 & 16 & 81 & 256 & 625 \\
    1 & 64 & 729 & 4096 & 15625 \\
    1 & 256 & 6561 & 65536 & 390625 \\
    1 & 1024 & 59049 & 1048576 & 9765625
  \end{pmatrix}
  \begin{pmatrix}
    M_1 \\ M_2 \\ M_3 \\ M_4 \\ M_5
  \end{pmatrix}
  =
  \begin{pmatrix}
    8{,}489{,}882{,}419{,}200 \\
    27{,}534{,}753{,}792{,}000 \\
    144{,}919{,}756{,}800{,}000 \\
    1{,}040{,}449{,}536{,}000{,}000 \\
    9{,}364{,}045{,}824{,}000{,}000
  \end{pmatrix}.
\]
The coefficient matrix is invertible: after factoring $1,4,9,16,25$
from the five columns, we obtain the standard Vandermonde matrix
associated to the values $\{1,4,9,16,25\}$. The remaining multiplicity
is determined by the total count
\begin{equation}\label{eq:total}
  M_0
  =
  |\Lambda_{10}|-2\sum_{j=1}^{5}M_j.
\end{equation}

Since the coefficient matrix is invertible, the moment identities
determine a unique solution $(M_1,\dots,M_5)$; because the moments
themselves are independent of $x_0$, so are $M_1,\dots,M_5$. Then
$M_0$ is also independent of $x_0$ by~\eqref{eq:total}. We record the
resulting distribution of inner products in the following proposition:

\begin{proposition}\label{prop:dist}
  Let $L$ be an extremal Type~II lattice of rank~$72$, and let
  $x_0\in\Lambda_8(L)$. The multiplicities $M_j=M_j(L;x_0)$ are
  independent of the choice of~$x_0$ and are given by
  \[
    \begin{aligned}
      M_0 &= 5{,}723{,}204{,}788{,}224, \\
      M_1 &= 3{,}708{,}252{,}979{,}200, \\
      M_2 &= 975{,}222{,}374{,}400, \\
      M_3 &= 93{,}206{,}937{,}600, \\
      M_4 &= 2{,}595{,}532{,}800, \\
      M_5 &= 13{,}959{,}168.
    \end{aligned}
  \]
  In particular, all six multiplicities are positive integers, and 
	$M_0+2(M_1+M_2+M_3+M_4+M_5)=|\Lambda_{10}|$ as expected.
\end{proposition}

Noting in particular from Proposition~\ref{prop:dist} that $M_4>0$, we
can then obtain Theorem~\ref{thm:main} using the well-known identity
\begin{equation}\label{eq:parallelogram}
  \langle x+x',x+x'\rangle
  =
  \langle x,x\rangle+2\langle x,x'\rangle+\langle x',x'\rangle.
\end{equation}

\begin{proof}[Proof of Theorem~\ref{thm:main}]
We fix $x_0\in\Lambda_8(L)$. By Proposition~\ref{prop:dist}, we have
$M_4=M_4(L;x_0)>0$, which means that there exists
$y\in\Lambda_{10}(L)$ with $\langle x_0,y\rangle=4$. Then, using
\eqref{eq:parallelogram}, we obtain
\[
  \langle y-x_0,\,y-x_0\rangle
  =
  10-2\cdot 4+8
  =
  10,
\]
so $y-x_0\in\Lambda_{10}(L)$. It follows that
\[
  x_0=y-(y-x_0)\in\mathcal{L}_{10}(L).
\]
Hence, we see that every minimal vector $x_0\in\Lambda_8(L)$ is the 
difference of two norm-$10$ vectors in $L$, and so
\[
  \Lambda_8(L)\subseteq \mathcal{L}_{10}(L).
\]
Therefore, we have
\begin{equation}\label{eq:8in10}
  \mathcal{L}_8(L)\subseteq \mathcal{L}_{10}(L).
\end{equation}

Finally, we know from Theorem~\ref{thm:prior} that
$\mathcal{L}_8(L)=L$; the claimed result thus follows
from~\eqref{eq:8in10}.
\end{proof}

\section{Remarks}

\subsection*{Why the inner product value $\min(L)/2$ is critical}

Suppose more generally that $L$ is any Type~II lattice with minimal norm
$m_0=\min(L)$, and let $x_0\in\Lambda_{m_0}(L)$ and
$y\in\Lambda_{m_0+2}(L)$. If
$\langle x_0,y\rangle=\varepsilon m_0/2$ for some
$\varepsilon\in\{\pm1\}$, then the identity~\eqref{eq:parallelogram}
yields
\[
  \langle y-\varepsilon x_0,\,y-\varepsilon x_0\rangle
  =
  (m_0+2)-2(m_0/2)+m_0
  =
  m_0+2.
\]
Thus, if the multiplicity corresponding to the inner product value
$m_0/2$ is positive, then $x_0\in\mathcal{L}_{m_0+2}(L)$.

The mechanism just described is present in all ranks; what varies is
which moment relations are available. In ranks~$40$ and~$80$, Elkies and
Kominers~\cite{EK09} used weighted-theta relations in degrees $\{2,6\}$
and $\{2,4,6,10\}$, respectively, together with the total-count
equation---but those relations are not quite enough to determine all
multiplicities outright. Thus, Elkies and Kominers instead argued by
contradiction: assuming $x_0\notin\mathcal{L}_{m_0+2}$ forces
$M_{m_0/2}=0$, and the remaining multiplicities, including~$M_0$,
satisfy a homogeneous linear system whose solution space is
one-dimensional. The resulting projective solution has a negative entry,
which is impossible. In rank~$72$, by contrast, the five even-degree
moments $\{2,4,6,8,10\}$ exactly match the five unknowns
$M_1,\dots,M_5$, and then $M_0$ follows from the total count.

\subsection*{Geometric consequences}

The positivity of all six multiplicities in Proposition~\ref{prop:dist}
implies that for every minimal vector $x_0$ in an extremal Type~II
lattice of rank~$72$, there exist vectors $y$ of norm~$10$ achieving
each inner product value
\[
  \langle x_0,y\rangle\in\{0,\pm1,\pm2,\pm3,\pm4,\pm5\}.
\]
In particular, $M_5>0$ means that there exist pairs $(x_0,y)$ with
$\langle x_0,y\rangle=\pm5$, for which
\[
  \langle y\mp x_0,\,y\mp x_0\rangle=8.
\]
Thus $y\mp x_0$ is again a minimal vector, giving structural information
about the interaction between the shells $\Lambda_8$ and $\Lambda_{10}$
that may be useful in the study of extremal lattices in this dimension.

\subsection*{Open questions}

It is natural to ask whether an extremal Type~II lattice $L$ of rank~$72$
is also generated by shells farther from the minimum---for instance, by
$\Lambda_{12}(L)$ alone. If one fixes $x_0\in\Lambda_8(L)$ and counts
vectors $y\in\Lambda_{12}(L)$ by the value of $\langle x_0,y\rangle$,
then the admissible inner products are
\[
  0,\pm1,\pm2,\pm3,\pm4,\pm5,\pm6.
\]
Thus there are seven multiplicities, but only six basic equations coming
from the total count and the even moments of degrees $2$, $4$, $6$, $8$,
and~$10$. An additional relation---for example from weighted theta
series---would therefore be needed to determine the full distribution by
this method, and may be needed to force the critical multiplicity
directly. We leave the systematic exploration of higher-shell generation
for future work.

\providecommand{\bysame}{\leavevmode\hbox to3em{\hrulefill}\thinspace}
\providecommand{\MR}{\relax\ifhmode\unskip\space\fi MR }
\providecommand{\MRhref}[2]{%
  \href{http://www.ams.org/mathscinet-getitem?mr=#1}{#2}
}
\providecommand{\href}[2]{#2}

\end{document}